\documentclass[11pt]{amsart}
\usepackage[margin=1.1in]{geometry}
\usepackage{amsmath,amssymb,amsthm}
\usepackage{hyperref}
\newtheorem{example}{Example}

\theoremstyle{plain}
\newtheorem{theorem}{Theorem}[section]
\newtheorem{lemma}[theorem]{Lemma}
\newtheorem{proposition}[theorem]{Proposition}
\newtheorem{corollary}[theorem]{Corollary}
\theoremstyle{remark}
\newtheorem{remark}[theorem]{Remark}

\newcommand{\Z}{\mathbb{Z}}
\newcommand{\Q}{\mathbb{Q}}
\newcommand{\F}{\mathbb{F}}
\newcommand{\OK}{\mathcal{O}_K}
\newcommand{\C}{\mathfrak{C}}
\newcommand{\dd}{\mathfrak{d}}
\newcommand{\m}{\mathfrak{m}}
\newcommand{\vm}{v_{\m}}

\title{A Local Approach to Monogenity with an Application to Lenny Jones' Conjecture}
\author{Michail Karatarakis}
\address{Radboud University, Nijmegen, The Netherlands
}
\email{michail.karatarakis@ru.nl}

\author{Sumandeep Kaur}
\address[Sumandeep Kaur\footnote{Corresponding author}]{Department of Mathematics, Shanghai University, China} 
\email{suman@shu.edu.cn}

\date{}
\subjclass [2010]{11R04, 11R09.}
\keywords{Monogenic, Index, polynomial.}

\begin{document}

\begin{abstract}
The study of monogenic polynomials is a classical problem in algebraic number theory. Existing criteria for deciding whether a polynomial is monogenic typically rely on discriminant computations together with methods such as Dedekind's criterion, Newton polygons, or valuation-theoretic techniques. In this paper, we develop a general local criterion for the $p$-maximality of orders generated by roots of arbitrary monic irreducible polynomials. As an application, we apply this to irreducible polynomials of the type $f(X)=X^n+A(BX+1)^m,$
where $1\le m<n$, $\gcd(n,mB)=1$, and $A,~B\in \Z\setminus\{0\}$. We show that  $f$ is monogenic
if and only if both $A 
\quad\text{and}\quad n^n+(-1)^{n+m}B^n(n-m)^{\,n-m}m^mA$
are square-free. This provides a new proof of the main theorem of \cite{KK}, thereby proving Lenny Jones' conjecture \cite[Conjecture 4.1]{LJ}. Furthermore, we obtain explicit infinite families of irreducible non-monogenic polynomials, including trinomial, quadrinomial, and power-compositional families.
\end{abstract}
\maketitle
\section{Introduction}

Let $f(X)$ be a monic irreducible polynomial with integer coefficients, let $\theta$ be a root of
$f$, and let  $K=\mathbb Q(\theta)$. A classical problem in algebraic number
theory is to determine whether the order $\mathbb Z[\theta]$ coincides with the
ring of algebraic integers $\mathcal O_K$ of $K$ (\cite{Gaal,GR},\cite{gas},\cite{GY},\cite{LJ, LJ1, LJ2, LJJ, HS1}). Equivalently, one can ask conditions under which the polynomial $f$ defines a monogenic field. This problem has attracted
continuous attention from many mathematicians because the existence of a power integral basis is closely
connected with the arithmetic structure of number fields, including
discriminants, ramification, and ideal factorization [\cite{Mar, Nar}].

The discriminant of $f$ and the discriminant $d_K$ of $K$ are related by the classical identity
\[
\operatorname{disc}(f)
=
[\mathcal O_K:\mathbb Z[\theta]]^{2}d_K.
\]
This identity shows that determining whether $K$ is monogenic reduces to determining the index
$[\mathcal O_K:\mathbb Z[\theta]].$
Consequently, many existing approaches in the literature first determine the finite set of possible index divisors, usually with the help of the discriminant of the defining polynomial. The local behaviour at the primes dividing the index is then analyzed using Dedekind's criterion or other techniques. When the degree of the polynomial increases, the problem becomes more difficult to handle, leading to the development of more sophisticated methods such as Ore's index theorem, Newton polygon techniques, higher-order Newton polygons, the Montes algorithm, and valuation-theoretic approaches.
These methods have been successfully applied to determine monogenity criteria for many families of polynomials \cite{LJ, LJ1, LJ2, LJJ, HS1}. Although the polynomial families studied in the literature vary considerably, the corresponding proofs usually follow a similar pattern. Consequently, the local analysis must be carried out separately for each new polynomial family and despite their effectiveness, these methods remain largely family dependent. Whenever a new parametric family is considered, one usually has to derive new discriminant identities, construct new Newton polygons, or perform fresh local computations.



This naturally leads to the following question:\\ Can one formulate a general local criterion whose hypotheses can be verified directly for different polynomial families, thereby avoiding the need to develop a new local argument in each case?

The purpose of this paper is to answer this question in the affirmative. We develop a general local criterion for
deciding $p$-maximality that is independent of the defining polynomial.
Our starting point is the observation that for a wide range of polynomial
families, the essential local obstruction is governed by the behavior of
repeated roots modulo $p$. We show that this phenomenon admits a uniform
description, leading to a necessary and sufficient criterion for
$p$-maximality that is formulated for an arbitrary monic irreducible
polynomial in $\mathbb Z[X]$. Consequently, the general local theory is
developed, which is independent of any particular polynomial
family.

The significance of this approach is that it separates the general local
theory from the arithmetic of individual polynomial families. More
precisely, the criterion depends only on the local behaviour of repeated
roots modulo $p$, together with an explicit congruence obtained from their
Taylor expansion and a corresponding B\'ezout identity over
$\mathbb F_p$. Once these conditions have been established, the proof for a
particular polynomial family reduces to verifying explicit local
hypotheses, while the ideal-theoretic arguments remain unchanged. Thus the
same local criterion can be applied uniformly to different polynomial
families without constructing a new local argument in each case.

As an application of this work, we apply this criterion to the irreducible polynomials of the type $f(X)=X^n+A(BX+1)^m,$
where $1\le m<n$, $\gcd(n,mB)=1$ and $A,~B\in \Z\setminus \{0\}$. We yields a discriminant-free proof of the monogenity criterion of $f$, thereby recovering the theorems of \cite{KK}, and proving Lenny Jones' conjecture \cite[Conjecture 4.1]{LJ}.


\section{Preliminaries}

Throughout the paper, let $K=\mathbb Q(\theta)$ be a number field of degree $n\geq 2$ 
where $\theta$ is a root of a monic irreducible polynomial
$f(X)\in\mathbb Z[X]$. We write $\mathcal O_K$ for the ring of integers of
$K$ and  $I(\theta):=[\mathcal O_K:\mathbb Z[\theta]]$, for
the index of the order $\mathbb Z[\theta]$ in $\mathcal O_K$. Thus $K$ is
monogenic if  $I(\theta)=1$.

The conductor of the order $\mathbb Z[\theta]$ is the largest ideal of
$\mathcal O_K$ contained in $\mathbb Z[\theta]$, defined as
\[
\C_\theta
=
\{\,x\in\mathcal O_K : x\mathcal O_K\subseteq\mathbb Z[\theta]\,\}.
\]
The conductor plays a fundamental role in the study of orders and their
indices (\cite[Chapter~III]{Mar}, \cite[Chapter~I, \S4]{Nar}).

For a prime number $p$, the order $\mathbb Z[\theta]$ is said to be
{$p$-maximal} if
$p\nmid I(\theta).$
Equivalently, the localization
$\mathbb Z[\theta]_{(p)}$ is the maximal order of $K$ at the prime $p$.
Since
\[
I(\theta)=\prod_{p}p^{v_p(I(\theta))},
\]
the order $\mathbb Z[\theta]$ is maximal if and only if it is
$p$-maximal for every prime $p$.

The following two classical results will be used throughout the paper. The
first is the conductor-different identity (\cite[Chapter~III, Proposition~2.7]{Neukirch}).

\begin{proposition}\label{prop:conductor-different}
Let $\C_\theta$ denote the conductor of the order
$\mathbb Z[\theta]$ in $\mathcal O_K$, and let
$\mathfrak D_K=\mathfrak D_{\mathcal O_K/\mathbb Z}$ be the different of
$K$. Then
\begin{equation}\label{eq:cd}
\C_\theta\,\mathfrak D_K=(f'(\theta)).
\end{equation}
In particular,
$f'(\theta)\in\C_\theta .$
\end{proposition}

The second result is the classical lower bound for the different in terms of
the ramification index 
(\cite[Chapter~III, Proposition~2.8]{Neukirch}).

\begin{proposition}\label{prop:different}
Let $\mathfrak m$ be a prime ideal of $\mathcal O_K$ lying above a
prime $p$, and suppose that
$p\mathcal O_K=\mathfrak m^{\,e}\mathfrak a,~\mathfrak m\nmid\mathfrak a.$
Then
$\mathfrak m^{\,e-1}\mid\mathfrak D_K.$
\end{proposition}

The following two lemmas will be used later in the proof of our main results.
\begin{lemma}\label{lem:pmax}
The order $\mathbb Z[\theta]$ is $p$-maximal if and only if $\C_\theta+p\mathcal O_K=\mathcal O_K.$
\end{lemma}

\begin{proof}
Let $G=\mathcal O_K/\mathbb Z[\theta]$ be a finite abelian group, and let
$N=|G|$.

Firstly, suppose that
\(
\C_\theta+p\mathcal O_K=\mathcal O_K.
\)
So, there exist $c\in\C_\theta$ and $u\in\mathcal O_K$ such that
$1=c+pu$. For every $x\in\mathcal O_K$, we have $x=cx+pux,$
where $cx\in\mathbb Z[\theta]$. Hence,
$\mathcal O_K=\mathbb Z[\theta]+p\mathcal O_K,$
and hence $G=pG$. Since $G$ is finite, its $p$-primary component is trivial, and
therefore $p\nmid N$.

Now, suppose that $p\nmid N$. Since $NG=0$, we have
$N\mathcal O_K\subseteq\mathbb Z[\theta],$
so $N\in\C_\theta$. By B\'ezout's identity, there exist
$a,b\in\mathbb Z$ such that
$aN+bp=1.$
Hence
$1\in\C_\theta+p\mathcal O_K,$
which completes the proof.
\end{proof}

\begin{lemma}\label{lem:intmem}
Let $\mathfrak m$ be a maximal ideal of $\mathcal O_K$ lying above a prime $p$. Then, for every $t\in\mathbb Z$,
\[
t\in\mathfrak m
\iff
p\mid t.
\]
\end{lemma}

\begin{proof}
Since $\mathfrak m\cap\mathbb Z$
is a non-zero prime ideal of $\mathbb Z$ containing $p$, it follows that
$\mathfrak m\cap\mathbb Z=p\mathbb Z,$
which proves the claim.
\end{proof}

The following result is well known (\cite[Chapter~III]{Neukirch}).

\begin{proposition}\label{prop:eis}
If the minimal polynomial $f$ of $\theta$ is Eisenstein at a prime $p$, then
$\mathbb Z[\theta]$ is $p$-maximal.
\end{proposition}

\section{Local Criteria for \texorpdfstring{$p$}{p}-Maximality at a Double Root Modulo \texorpdfstring{$p$}{p}}
In this section, we establish the main theoretical results of the paper.
Building on the preliminary results of the previous section, we develop a
general criterion for deciding the $p$-maximality of the order
$\mathbb Z[\theta]$. The criterion is formulated for an arbitrary monic
irreducible polynomial and serves as the basis for all applications in the
subsequent sections.\\

Fix a prime $p$ and an integer $r$. Consider the second order Taylor
decomposition of $f$ at $r$. There exists $g \in \Z[X]$ with
\begin{equation}\label{eq:taylor2}
  f \ =\ f(r) \ +\ f'(r)\,(X - r) \ +\ (X-r)^2 g,
\end{equation}
and $g$ is monic of degree $n-2$ with
\begin{equation}\label{eq:gr}
  2\,g(r) \ =\ f''(r).
\end{equation}

Evaluating~\eqref{eq:taylor2} at $\theta$ and writing
$ \pi = \theta - r, ~S = g(\theta),$
we obtain the {following relation} in $\OK$:
\begin{equation}\label{eq:fund}
  \pi^2 S \ =\ -\bigl(f(r) + f'(r)\,\pi\bigr).
\end{equation}

\begin{theorem}\label{thm:suff}
Let $f\in\mathbb Z[X]$ be a monic irreducible polynomial, and let
$\theta$ be a root of $f$. Let $r \in \Z$ satisfy
\begin{enumerate}
\item[(i)] $p^2 \nmid f(r)$;
\item[(ii)] $p \nmid f''(r)$;
\item[(iii)] every maximal ideal $\m$ of $\OK$ with $\C_\theta \subseteq \m$ and $p \in \m$
contains $\theta - r$.
\end{enumerate}
Then $\Z[\theta]$ is $p$-maximal.
\end{theorem}


\begin{proof}
Suppose that $\Z[\theta]$ is not $p$-maximal. By Lemma~\ref{lem:pmax},
\[
\C_\theta+p\,\OK\neq\OK,
\]
so there exists a maximal ideal $\m$ of $\OK$ containing
$\C_\theta+p\,\OK$. By hypothesis~$\it{(iii)}$,
$\pi:=\theta-r\in\m.$

Consider the first-order Taylor expansion of $f'$ at $r$, i.e.,
\begin{equation*}\label{eq:taylor1}
f'=f'(r)+(X-r)g_1,\qquad g_1(r)=f''(r),
\end{equation*}
and set
$S=g(\theta),~ T=g_1(\theta).$
Then \eqref{eq:fund} holds and
\begin{equation}\label{eq:fprime}
f'(\theta)=f'(r)+\pi T.
\end{equation}

Since $f'(\theta)\in\C_\theta\subseteq\m$ and $\pi\in\m$,
identity~\eqref{eq:fprime} implies that $f'(r)\in\m$. Hence
$p\mid f'(r)$ by Lemma~\ref{lem:intmem}. Equation~\eqref{eq:fund}
then shows that $f(r)\in\m$, and therefore $p\mid f(r)$. Write
$f(r)=pt,\text{ where }p\nmid t, \text{ and }
f'(r)=pk.$

We next show that $S,T\notin\m$. Indeed, for every
$q\in\Z[X]$, we have 
$q(\theta)-q(r)\in(\theta-r)\subseteq\m.$
Thus $q(\theta)\in\m$ would imply $q(r)\in\m$, or equivalently,
$p\mid q(r)$ by Lemma~\ref{lem:intmem}. Since
\[
g_1(r)=f''(r),\qquad 2g(r)=f''(r),
\]
hypothesis~{\it{(ii)}} yields
\[
p\nmid g(r)\qquad\text{and}\qquad p\nmid g_1(r),
\]
and consequently
$S,T\notin\m.$ Now set
$\sigma=t+k\pi.$
Then \eqref{eq:fund} becomes
\begin{equation}\label{eq:idideal}
\pi^2S=-p\sigma,
\qquad\text{or equivalently}\qquad
(\pi)^2(S)=(p)(\sigma)
\end{equation}
as ideals of $\OK$. Furthermore, $\sigma\notin\m$. 
Also,
$\pi\neq0$ because $\deg f\ge2$. Let
\[
w=\vm((\pi)),\qquad
e=\vm((p)).
\] Since
$\vm((S))=\vm((\sigma))=0$, taking $\m$-adic valuations
in~\eqref{eq:idideal}, we have  $2w=e.$ We claim that
$f'(\theta)\notin\m^{\,w+1}.$
By~\eqref{eq:fprime},
$\pi T=f'(\theta)-pk.$
Since
$p\in\m^{\,e}
=\m^{2w}
\subseteq\m^{\,w+1},$
the assumption $f'(\theta)\in\m^{\,w+1}$ would imply
$\pi T\in\m^{\,w+1}$. On the other hand,
$\vm((\pi T))
=\vm((\pi))+\vm((T))
=w,$
because $T\notin\m$. This contradicts
$\pi T\in\m^{\,w+1}$.

Finally, since $\m^{\,e}\mid p\,\OK,$
Proposition~\ref{prop:different} gives
$\m^{\,e-1}\mid\dd_K.$
Moreover,
$\m\mid\C_\theta$
by construction. Hence, using the conductor--different
identity~\eqref{eq:cd},
\[
\m^{\,e}
=\m^{\,1+(e-1)}
\mid
\C_\theta\dd_K
=(f'(\theta)).
\]
Therefore
$f'(\theta)\in\m^{\,e}
=\m^{2w}
\subseteq\m^{\,w+1},$
which is a contradiction. Thus $\Z[\theta]$ is $p$-maximal.

\end{proof}

\begin{theorem}\label{thm:nec}
Let $f\in\mathbb Z[X]$ be a monic irreducible polynomial, and let
$\theta$ be a root of $f$. 
Suppose $p^2 \mid f(r)$ and $p \mid f'(r)$. Then $p \mid I[\theta]$; in particular
$\Z[\theta] \ne \OK$.
\end{theorem}

\begin{proof}
Write $f(r) = p^2 t,~f'(r) = pk$, and set
$y =\pi S = (\theta - r)\,~g(\theta) \ \in\ \Z[\theta].$
Multiplying~\eqref{eq:fund} by $S$, we have $y^2 = -f(r)S - f'(r)y$, i.e.
\begin{equation*}\label{eq:quad}
  y^2 + f'(r)\,y + f(r)\,S = 0,
  \qquad\text{ and hence}\qquad
  \Bigl(\frac{y}{p}\Bigr)^{\!2} + k\Bigl(\frac{y}{p}\Bigr) + t\,S \ =\ 0.
\end{equation*}
Thus $z := y/p \in K$ satisfies a {monic} quadratic polynomial over $\OK$, so $z \in \OK$.

We claim $z \notin \Z[\theta]$. Suppose $z = c(\theta)$ for some $c \in \Z[X]$; without loss of generality, we may assume $\deg c < n$. Consider
$ W \ =\ (X - r)\,g \ \in\ \Z[X],$
which by~\eqref{eq:taylor2} is a monic polynomial of degree $n-1$, and satisfies $W(\theta) = y = p\,c(\theta)$.
Then $W - p\,c$ vanishes at $\theta$ and has degree $< n = \deg f$, so $W = p\,c$ identically.
Comparing leading coefficients, we get a contradiction. Thus $z \notin \mathbb Z[\theta]$ but $pz\in\mathbb Z[\theta]$. Therefore the class of $z$ in $\mathcal O_K/\mathbb Z[\theta]$ has order $p$, and hence, the proof is complete.
\end{proof}

\begin{remark}
    The search of a suitable integer \(r\) satisfying the hypotheses of  Theorem~\ref{thm:nec} can  be restricted to values of $r$ such that $0\le r<p^2$  as both $f(r)\bmod p^2$ and $f'(r)\bmod p$ depend only on
$r\bmod p^2$.

\end{remark}

\begin{remark}
The converse of Theorem~\ref{thm:nec} is not true. 
For example, consider $
f(X)=X^4+2X^3+5X^2+4X+7.
$ Here $f$ does not satisfy the hypotheses of Theorem 3.2, but it is non-monogenic.

\end{remark}
The following theorem is a generalized version of Theorem~\ref{thm:nec}.
\begin{theorem}\label{thmm:necfactor}
Let $f\in\Z[X]$ be monic irreducible polynomial of degree $n$ with a root $\theta$, and let $p$ be a
prime. Let $\bar\varphi$ be a repeated irreducible factor of $\bar f$ in
$\F_p[X]$ of degree $d$, let $\varphi\in\Z[X]$ be a monic lift of $\bar{\varphi}$, and  $\alpha$ be a root of
$\varphi$. Suppose that $f=\sum_k a_k\varphi^{\,k}$  where $\deg a_k<d$  is the $\varphi$-adic expansion of $f$.
 If
$p^2\mid a_0 ,$ then $p\mid I(\theta)$; in particular $f$ is non-monogenic.
\end{theorem}
\begin{proof} It is easy to check that  $p\mid a_1$ and $p\mid f'(\alpha)$. 
Write $a_0=p^2t$, $a_1=pk$ with $t,k\in\Z[X]$, and $G=\sum_{k\ge2}a_k\varphi^{\,k-2}$. Then
\begin{equation}\label{eq:obstruction}
f=\varphi^2G+p\,k\varphi+p^2t .
\end{equation}
Set $W=\varphi G$, which is monic of degree $n-d$, and
\[
y=W(\theta)=\varphi(\theta)G(\theta)\in\Z[\theta],
\qquad
z=\frac{y}{p}\in K=\Q(\theta).
\]
We show that $z\in\OK\setminus\Z[\theta]$. Since $pz=y\in\Z[\theta]$, the
class of $z$ is an element of order $p$ of the finite group
$\OK/\Z[\theta]$, and thus $p\mid I(\theta)$. Evaluating \eqref{eq:obstruction} at
$\theta$ and using $f(\theta)=0$, we have
$\varphi(\theta)^2G(\theta)=-p\bigl(k(\theta)\varphi(\theta)+p\,t(\theta)\bigr)$.
Multiplying by $G(\theta)$, and then dividing by $p^2$, we get
$z^2+k(\theta)\,z+t(\theta)G(\theta)=0 .$
Thus $z\in\OK$.

It remains to prove that $z\not\in \Z[\theta]$. The set $\{1,\theta,\dots,\theta^{\,n-1}\}$ form
a $\Z$-basis of $\Z[\theta]$, and $\deg W<n$, so the coefficients of $W$
are the coordinates of $y$ in that basis.  If $z=y/p\in \Z[\theta]$,
all of them would be divisible by $p$. So $z\not\in \Z[\theta]$ because $W$ is monic.
\end{proof}

\section{Applications to Lenny Jones' Conjecture}
The following theorem was established by Kaur and Kumar \cite{KK}. We shall derive it as an application of the local criteria developed in the previous section.
\begin{theorem}\label{th1}
Let $A,B\in\Z$ with $B\neq0$, and let $n,m$ be positive integers satisfying $1\le m<n,~ n>2.$
Assume that $\gcd(n,mB)=1$. Let
$f(X)=X^n+A(BX+1)^m$
be irreducible over $\Q$, and let $D=n^n+(-1)^{n+m}B^n(n-m)^{\,n-m}m^mA.$
Then $f$ is monogenic, if and only if both $A$ and $D$ are square-free.
\end{theorem}
The following corollary follows from the above theorem. It is conjectured  by L. Jones in \cite[Conjecture 4.1]{LJ}.
\begin{corollary}\label{cor}
	Let $p$ be a prime number, $n, m$ and  $B$ be positive integers with $1\le m\le n-1,~n>2$ and $\gcd(n,mB)=1.$ Then $f(X)=X^n+p(BX+1)^m$ is monogenic if and only if   $n^n+(-1)^{n+m}B^n(n-m)^{n-m}m^mp$ is square-free.
\end{corollary} 
Throughout this section, let $f(X)=X^n+A(BX+1)^m,$
where $A,B\in\Z\setminus\{0\}$, $1\le m<n$, and assume that $f$ is irreducible over $\Z$.
Let   \begin{align}D &= n^n + (-1)^{n+m} B^n (n-m)^{n-m} m^m A \ \in\ \Z. \label{eq:D}\end{align} Direct computation gives, for
$a \in \Z$,
\begin{align}
  f(a)   &= a^n + A\,(Ba+1)^m, \label{eq:fa}\\
  f'(a)  &= n\,a^{n-1} + mAB\,(Ba+1)^{m-1}, \label{eq:fpa}\\
  f''(a) &= n(n-1)\,a^{n-2} + m(m-1)AB^2\,(Ba+1)^{m-2}. \label{eq:fppa}
\end{align}

\begin{lemma}\label{lem:cong}
Let $f,~D$ be as above and $a, N \in \Z$ with $N \mid B(n-m)a + n$. Then
\[
  \bigl(B(n-m)\bigr)^n f(a) \ \equiv\ (-1)^n D \pmod N.
\]
\end{lemma}

\begin{proof}
Set $z = B(n-m)$, so that $za \equiv -n \pmod N$. Note that,
\begin{equation}\label{eq:aux}
  (n-m)(Ba+1) \ =\ za + (n-m) \ \equiv\ -n + (n-m) \ =\ -m \pmod N.
\end{equation}
So, keeping in mind \eqref{eq:fa}, we have
\begin{align*}
  z^n f(a)
  &= (za)^n + A\,B^n (n-m)^n (Ba+1)^m\\
  &= (za)^n + A\,B^n (n-m)^{n-m}\bigl((n-m)(Ba+1)\bigr)^m\\
  &\equiv (-n)^n + A\,B^n (n-m)^{n-m}(-m)^m \pmod N\\
  &= (-1)^n\Bigl(n^n + (-1)^{n+m} B^n (n-m)^{n-m} m^m A\Bigr) \ =\ (-1)^n D. \qedhere
\end{align*}
\end{proof}

\begin{corollary}\label{cor:cong}
If $\gcd\bigl(N,\ B(n-m)\bigr) = 1$, then $N \mid f(a)$ if and only if $N \mid D$.
\end{corollary}

\begin{lemma}\label{lem:bezout}
If $p \nmid B(n-m)$, there exists $a \in \Z$ with $p^2 \mid B(n-m)\,a + n$.
\end{lemma}

\begin{proof}
Since $p\nmid B(n-m)$, the congruence $B(n-m)a\equiv-n\pmod{p^2}$ has a solution.
\end{proof}


\begin{lemma}\label{lem:coprime}
Let $f,~D$ be as above and assume  that $\gcd(n,\ mB) = 1$. If $p \nmid A$ and $p \mid D$, then $p \nmid n$, $p \nmid B$,
$p \nmid m$, and $p \nmid n-m$.
\end{lemma}

\begin{proof}

First suppose that $p \mid n$. Since $\gcd(n,mB)=1$, we have
$p \nmid m$ and $p \nmid B$. 
As $p\mid D$ and $p\mid n^n$, it follows that
$p\mid B^n(n-m)^{\,n-m}m^mA.$
Hence $p$ divides one of $B$, $n-m$, $m$, or $A$. The first, third, and fourth
possibilities contradict $p\nmid B$, $p\nmid m$, and $p\nmid A$, respectively.
If $p\mid n-m$, then together with the fact that $p\mid n$, we obtain that $p\mid m$, again a
contradiction. Therefore $p\nmid n$.

Now suppose that $p$ divides one of $B$, $m$, or $n-m$. Since $p\nmid n$, we
have $p\nmid n^n$. On the other hand,
$D-n^n=(-1)^{n+m}B^n(n-m)^{\,n-m}m^mA$
is divisible by $p$, and hence so is $D-n^n$. As $p\mid D$, this implies
$p\mid n^n$, contradicting $p\nmid n$. Thus
$p\nmid B$, $p\nmid m$, and $p\nmid n-m$.
\end{proof}

\begin{lemma}\label{lem:sqA}
Assume $ \gcd(n,\ mB) = 1$. If $p \mid A$ and $p^2 \mid D$, then $p^2 \mid A$.
\end{lemma}

\begin{proof}
From $p \mid A$ and $p \mid D$, we get $p \mid n^n$, hence $p \mid n$; then the hypothesis
gives $p \nmid m, B$, and consequently $p \nmid n-m$. Since $n \ge 2$ we get $p^2 \mid n^n$, so
$p^2$ divides $D - n^n = (-1)^{n+m}B^n (n-m)^{n-m} m^m A$, all of whose factors except $A$ are
coprime to $p$. Hence $p^2 \mid A$.
\end{proof}


The next proposition is stated for an arbitrary field; so it applies simultaneously to $\F_p$ {and} to every residue field $\OK/\m$ with
$p \in \m$.

Let $F$ be a field in which the image of $\Z$ has kernel $p\Z$. Assume that  
\begin{equation*}
  \gcd(n,\ mB) = 1,~~ p\nmid A. \tag{$*$}\label{eq:str}
\end{equation*} and write $\bar f$
for the image of $f$ in $F[X]$, and $u = \bar B x + 1$ for $x \in F$.

\begin{proposition}\label{prop:residue}
Assume that \eqref{eq:str} holds. Let $x \in F$ be a common root of $\bar f$ and $\bar f'$. Then $x \ne 0$, $u \ne 0$,
$p \nmid n$, $p \nmid B$, $p \nmid m$, $p \nmid n-m$, and
\begin{equation}\label{eq:lin}
  (\bar n - \bar m)\,\bar B\,x \ =\ -\,\bar n.
\end{equation}
In particular; the coefficient $(\bar n - \bar m)\bar B$ is invertible, so $x$ is uniquely
determined by~\eqref{eq:lin} and lies in the prime field of $F$.
\end{proposition}

\begin{proof}
Since $\bar f(x)=0$, we have $x\neq0$, for otherwise
$\bar f(x)=\bar A\neq0$. Similarly, $u\neq0$, since otherwise
$\bar f(x)=x^n\neq0$. Combining \eqref{eq:fa} and \eqref{eq:fpa}, we obtain
\begin{align}
0
&=\bar n\,\bar f(x)-x\,\bar f'(x)\notag\\
&=\bar n\bigl(x^n+\bar Au^m\bigr)
-x\bigl(\bar n x^{n-1}+\bar m\bar A\bar Bu^{m-1}\bigr)\notag\\
&=\bar A u^{m-1}\bigl(\bar nu-\bar m\bar Bx\bigr).
\label{eq:collapse}
\end{align}
Since $\bar A\neq0$ and $u\neq0$, it follows that
$\bar nu=\bar m\bar Bx$, which is equivalent to
\eqref{eq:lin}.

We next prove that $p\nmid n$, $p\nmid B$, $p\nmid m$, and
$p\nmid(n-m)$. If $p\mid n$, then $\bar n=0$. Since $p\nmid A$ and
$\gcd(n,mB)=1$, we have $\bar A,\bar B,\bar m\neq0$.
Hence $\bar f'(x)=\bar m\bar A\bar Bu^{m-1}\neq0,$
contradicting $\bar f'(x)=0$. Thus $p\nmid n$.
If $p\mid B$, then \eqref{eq:lin} gives $0=-\bar n$, contradicting
$p\nmid n$. The case $p\mid(n-m)$ is identical. Finally, if
$p\mid m$, then $\bar m=0$, so \eqref{eq:lin} yields
$\bar Bx=-1$, that is, $u=0$, again a contradiction. Hence $(\bar n-\bar m)\bar B$ is invertible. Therefore
\eqref{eq:lin} uniquely determines $x$, and consequently
$x$ belongs to the prime field of $F$.
\end{proof}

\begin{proposition}\label{prop:Dzero}
Assume that \eqref{eq:str} holds. Let $x\in F$ satisfy
\eqref{eq:lin}, and suppose that
$p\nmid nBm(n-m).$
Then
$\bar f(x)=\bar f'(x)=0
~\Longleftrightarrow~
\bar D=0
\ \text{in }F.$
\end{proposition}

\begin{proof}
By \eqref{eq:lin}, $(\bar n-\bar m)u=-\bar m.$
Raising 
$(\bar n-\bar m)\bar Bx=-\bar n
~\text{and}~
(\bar n-\bar m)u=-\bar m$
to the powers $n$ and $m$, respectively, and arguing exactly as in the proof
of Lemma~\ref{lem:cong}, we have
\begin{equation}\label{eq:fieldcong}
\bigl((\bar n-\bar m)\bar B\bigr)^n\bar f(x)
=
(-1)^n\bar D.
\end{equation}
Since $(\bar n-\bar m)\bar B$ is invertible, it follows that
$\bar f(x)=0
~\Longleftrightarrow~
\bar D=0.$ It remains to prove that $\bar f(x)=0
~\Longrightarrow~
\bar f'(x)=0.$
From \eqref{eq:lin}, we have $x\neq0$, for otherwise $\bar n=0$.
Moreover, $(\bar n-\bar m)u=-\bar m$
shows that $u\neq0$, since $p\nmid m$ and $p\nmid(n-m)$. Multiplying
$\bar f'(x)$ by the nonzero element $xu$ and using
$\bar n u=\bar m\bar Bx$, we obtain
$\bar f'(x)\,xu
=
\bar nux^n+\bar m\bar A\bar Bxu^m
=
\bar nu\bigl(x^n+\bar Au^m\bigr)
=
\bar nu\,\bar f(x)
=
0.$
Hence $\bar f'(x)=0$, which completes the proof.
\end{proof}

\begin{proposition}\label{prop:tame}
Assume that \eqref{eq:str} holds. Let $x\in F$ satisfy $\bar f(x)=\bar f'(x)=0.$
Then $\bar f''(x)\neq0.$
\end{proposition}

\begin{proof}
Since $\bar f(x)=\bar f'(x)=0$, Proposition~\ref{prop:residue}
shows that \eqref{eq:lin} holds. As
$u=\bar Bx+1,$
equation~\eqref{eq:lin} yields
$(\bar n-\bar m)\bar Bx=-\bar n,$
from which we obtain
$-(\bar n-1)u+(\bar m-1)\bar Bx=1.$ Multiplying \eqref{eq:fppa} by the nonzero element $xu$ and using
$\bar f'(x)=0$, we obtain $\bar f''(x)\,xu
=
\bar m\,\bar A\,\bar B\,u^{m-1}.$
Since $\bar A$, $\bar B$, $\bar m$, and $u$ are nonzero, the right-hand side
does not vanish. As $xu\neq0$, we conclude that
$\bar f''(x)\neq0.$
\end{proof}


\section{Proof of Theorem \ref{th1}}\label{sec:proof}

\begin{proof}
Let $\theta$ be a root of $f(X)=X^n+A(BX+1)^m,$
and  $K=\Q(\theta)$. Let $p$ be a prime. Note that $f(0)=A
~\text{and}~
f'(0)=mAB.$

First, suppose that $p^2\mid A$. Since $p^2\mid f(0)
\quad\text{and}\quad
p\mid f'(0),$
Theorem~\ref{thm:nec} yields $p\mid I[\theta].$

Next assume that $p^2\mid D$ and $p\mid A$. By Lemma~\ref{lem:sqA},
we have $p^2\mid A$, so the previous case applies.

Finally, assume that $p^2\mid D$ and $p\nmid A$. By
Lemma~\ref{lem:coprime}, we have $p\nmid B(n-m)$. Let $a$ be as in
Lemma~\ref{lem:bezout}, so that $p^2\mid B(n-m)a+n.$
Applying Corollary~\ref{cor:cong} with $N=p^2$, we obtain
\[
p^2\mid f(a)
\iff
p^2\mid D.
\]
Hence $p^2\mid f(a)$. Moreover, the residue class $\bar a$ satisfies
\eqref{eq:lin} over $\F_p$. By Proposition~\ref{prop:Dzero},
$\bar f'(\bar a)=0,$
and therefore $p\mid f'(a).$
Applying Theorem~\ref{thm:nec} with $r=a$, we conclude that
$p\mid I[\theta].$

For the converse, we prove it by contraposition. Assume that
$p^2\nmid A
\quad\text{and}\quad
p^2\nmid D.$
We shall prove that $\Z[\theta]$ is $p$-maximal.

Suppose first that $p\mid A$. Since $p^2\nmid A$, every coefficient of
$f$ except the leading one is divisible by $p$, while the constant term
$A$ is not divisible by $p^2$. Hence $f$ is Eisenstein at $p$, and the
result follows from Proposition~\ref{prop:eis}.

Now assume that $p\nmid A$ and $p\mid D$. By Lemma~\ref{lem:coprime}, $p\nmid n B m(n-m).$
Let $a$ be as in Lemma~\ref{lem:bezout}, so that
$p^2\mid B(n-m)a+n.$ By Corollary~\ref{cor:cong},  
$p^2\nmid f(a),$
since $p^2\nmid D$. The residue class $\bar a$ satisfies
\eqref{eq:lin}. Since $p\mid D$, Proposition~\ref{prop:Dzero} implies
that
$\bar f(\bar a)=\bar f'(\bar a)=0.$
Hence Proposition~\ref{prop:tame} gives
$\bar f''(\bar a)\neq0,$
and therefore $p\nmid f''(a).$ Let $\m$ be a maximal ideal satisfying
$\C_\theta\subseteq\m
\quad\text{and}\quad
p\in\m,$
and put $F=\OK/\m$. Reducing modulo $\m$, we obtain $\bar f(\bar\theta)=0.$
Moreover,
$f'(\theta)\in\C_\theta\subseteq\m$
by the conductor-different identity~\eqref{eq:cd}, so
$\bar f'(\bar\theta)=0.$
Proposition~\ref{prop:residue} therefore shows that $\bar\theta$
satisfies~\eqref{eq:lin}. On the other hand, $p\mid B(n-m)a+n,$
so $\bar a$ also satisfies~\eqref{eq:lin} in $\F_p$. Since
$(\bar n-\bar m)\bar B$ is invertible in $\F_p$, the solution of
\eqref{eq:lin} is unique. Consequently,
$\bar\theta=\bar a
\quad\text{in }\F_p,$
or equivalently,
$\theta-a\in\m.$ Hence, all the hypotheses of Theorem~\ref{thm:suff} are therefore satisfied,
and hence $\Z[\theta]$ is $p$-maximal.

Finally, assume that $p\nmid A$ and $p\nmid D$. By
Lemma~\ref{lem:pmax}, it is enough to prove that no maximal ideal
contains $\C_\theta+p\OK$. Suppose, to the contrary, that there exists a
maximal ideal $\m$ containing $\C_\theta+p\OK$, and let
$F=\OK/\m$. As above,
$\bar f(\bar\theta)=\bar f'(\bar\theta)=0.
$
By Proposition~\ref{prop:residue}, $\bar\theta$ satisfies
\eqref{eq:lin}, and Proposition~\ref{prop:Dzero} therefore yields
$\bar D=0
\quad\text{in }F.$
By Lemma~\ref{lem:intmem}, this implies that $p\mid D$, contradicting
our assumption. Hence no such maximal ideal exists, and
$\Z[\theta]$ is $p$-maximal. This completes the proof.
\end{proof}

%
%
%



\section{Examples and infinite families}\label{sec:examples}



\begin{corollary}\label{cor:master}
Let $f=X^n+a_{n-1}X^{n-1}+\dots+a_1X+a_0\in\Z[X]$ be monic and irreducible
with $n\ge2$, let $\theta$ be a root of $f$, and let $p$ be a prime.  If
$p\mid a_1$ and $p^2\mid a_0$, then $f$ is
non-monogenic.
\end{corollary}

\begin{proof}
Note that $p\mid I(\theta)$ by  Theorem~\ref{thm:nec} as $f(0)=a_0$ and $f'(0)=a_1$.
\end{proof}

\begin{lemma}\label{lem:wild}
Let $f=X^n+a_{n-1}X^{n-1}+\dots+a_0\in\Z[X]$ be monic and irreducible with
root $\theta$, and let $p$ be a prime such that $p\mid n,~\text{and}~ p\mid a_i \text{ whenever } p\nmid i. $ If $p^2\mid f(r)$ for a single $r\in\Z$, then $p\mid I(\theta)$; in particular
$f$ is not monogenic.
\end{lemma}

\begin{proof}
By given hypothesis, every coefficient of
$f'=\sum_i i\,a_iX^{i-1}$ is divisible by $p$, that is $\bar f'=0$.  Hence
$p\mid f'(r)$ for every $r\in\Z$, and Theorem~\ref{thm:nec} applies.
\end{proof}

\subsection{Application to a Family of Trinomials}\label{ss:trinomial}

\begin{proposition}\label{prop:fam1}
Let $p\ne q$ be primes, $k\in\Z$, $n>m\geq 2$, and  let
$f(X)=X^n+qk\,X^m+p^2q$.
Then $f$
is non-monogenic.
\end{proposition}

\begin{proof}
Note that $f$ is irreducible over $\Q$ and $p\mid I(\theta)$ by Theorem~\ref{thm:nec}.  Hence $f$ is non-monogenic.
\end{proof}


\begin{example}
Consider $f(X)=X^{35}+2X^2+18.
$ Here $p=3$, $q=2$, $k=1$, $n=35$, and $m=2.$ Hence  $f$ is non-monogenic by above Proposition.

\end{example}

\begin{example} Let
$f(X)=X^{16}+10X^3+20.
$ Here $p=2$, $q=5$, $k=2$, $n=16$, and $m=3$. Hence $f$ is non-monogenic.
\end{example}

\begin{remark}\label{rem:translate}
Applying Corollary~\ref{cor:master} to $f(X+r)$, whose constant and linear
coefficients are $f(r)$ and $f'(r)$, returns Theorem~\ref{thm:nec} at a
general $r$; conversely, every application of Theorem~\ref{thm:nec} is an
application of Corollary~\ref{cor:master} to a translate of $f$. 
\end{remark}

\subsection{Application to a quadrinomial family}\label{ss:quadrinomial}

\begin{proposition}\label{prop:quadrinomial}
Let $n\ge3$ and $a,b,c\in\Z$ with $b^2=ac$, and let
$f(X)=X^n+aX^2+2bX+c$ be irreducible with root $\theta$.  If $p$ is a prime
with $p\mid c$ and $p\nmid a$, then $p\mid I(\theta)$; in particular $f$ is
non-monogenic.
\end{proposition}

\begin{proof}
Proof is an immediate consequence of Corollary~\ref{cor:master}.
\end{proof}

\begin{example}\label{ex:quadrinomial}
Let $f(X)=X^5+X^2+6X+9$.  Taking $p=3$, Proposition~\ref{prop:quadrinomial}
gives $3\mid I(\theta)$ and hence $f$ is non-monogenic.
\end{example}
\subsection{Application to power-compositional polynomials $g(X^p)$}\label{ss:omial}{}
\begin{remark}
Recall that for $c\in\Z$ the residue $c^{\,p}\bmod p^2$ depends only on $c$ modulo $p$,
since $(c+p)^p\equiv c^{\,p}\pmod {p^2}$.  Write
$\omega(c)\in\Z/p^2\Z$ for this residue; the $p$ values
$\omega(0),\dots,\omega(p-1)$ are the Teichm\"uller lifts of $\F_p$.
    \end{remark}
\begin{theorem}\label{thm:powcomp}
Let $p$ be a prime, let $g\in\Z[X]$ be monic, and let $f(X)=g(X^p)$ be
irreducible with root $\theta$.  If
\[
g\bigl(\omega(c)\bigr)\equiv0 \pmod{p^2}\qquad\text{for some }c\in\{0,\dots,p-1\},
\]
then $p\mid I(\theta)$; in particular $f$ is not monogenic.
\end{theorem}

\begin{proof}
Note that $f'(X)=p\,X^{p-1}g'(X^p)$, and
$f(c)=g(c^{\,p})\equiv g(\omega(c))\equiv0 \pmod{p^2}$.  Thus, $f$ is non-monogenic by Lemma~\ref{lem:wild} with $r=c$.
\end{proof}

\begin{example}\label{ex:powcomp}
Let $p=3$, so $\omega(0)=0$, $\omega(1)=1$, $\omega(2)=8$. Then
\begin{itemize}
\item[(i)] $g=Y^2+5Y+3$ has $g(1)=9$, so  
$3\mid I(\theta)$ and $X^6+5X^3+3$ is non-monogenic.
\item[(ii)] $g=Y^2-7Y+1$ has $g(8)=9$, so $X^6-7X^3+1$ is non-monogenic.
\end{itemize}
\end{example}

\subsection{Comparison with the classical route}\label{ss:comparison}

Fix a prime $p$ and a monic irreducible $f\in\Z[X]$ of degree $n$.  For applying Dedekind's criterion, one need to factor
$\bar f=\prod_i\bar g_i^{\,e_i}$ in $\F_p[X]$, lifts $\bar g=\prod_i\bar g_i$
and $\bar h=\bar f/\bar g$ to monic $g,h\in\Z[X]$, forms
$M=(gh-f)/p$, and computes $\gcd(\bar M,\bar g,\bar h)$. The order is
$p$-maximal if and only if this gcd is $1$.  The Newton polygon and Montes
methods require, in addition, $\phi$-adic expansions of $f$, the sides of the
associated polygons and their residual polynomials.

For example, consider $f=X^3+2X+22$ and $p=5$. Here
$\bar f=(X-1)^2(X+2),~
g=(X-1)(X+2),~
h=X-1,~
gh-f=-5X-20,~ M=-X-4,$
and $\bar M=-(X-1)$ in $\F_5[X]$, so
$\gcd(\bar M,\bar g,\bar h)=X-1\neq1$ and $5\mid I(\theta)$. In contrast, Theorem~\ref{thm:nec} immediately yields $5\mid I(\theta)$ with $r=1$.

\end{document}